\documentclass[a4paper,11pt,reqno]{amsart}
\usepackage{amsmath}
\usepackage{amssymb}
\usepackage{xcolor}

\usepackage{mathrsfs}
\renewcommand{\mathcal}{\mathscr}

\theoremstyle{plain} 
\newtheorem{theorem}{Theorem}[section]

\newtheorem{lemma}[theorem]{Lemma}
\newtheorem{corollary}[theorem]{Corollary}
\newtheorem{remark}[theorem]{Remark}

\makeatletter
    
    \@addtoreset{equation}{section}
  \makeatother

\title[Minimal random walk model of elephant type]{Limit theorems for the `laziest' minimal random walk model of elephant type}
\thanks{M.T. is partially supported by JSPS  Grant-in-Aid for Scientific Research (B) No. 19H01793 and (C) No. 19K03514.
}
\author{Tatsuya Miyazaki}
\address{Graduate School of Engineering Science, Yokohama National University, Yokohama, Japan}
\email{miyazaki-tatsuya-ct@ynu.jp}
\author{Masato Takei}
\address{Department of Applied Mathematics, Faculty of Engineering, Yokohama National University, Yokohama, Japan}
\email{takei-masato-fx@ynu.ac.jp}

\begin{document}

\begin{abstract}
We consider a minimal model of one-dimensional discrete-time random walk with step-reinforcement, introduced by Harbola, Kumar, and Lindenberg (2014): The walker can move forward (never backward), or remain at rest. For each $n=1,2,\cdots$, a random time $U_n$ between $1$ and $n$ is chosen uniformly, and if the walker moved forward [resp. remained at rest] at time $U_n$, then at time $n+1$ it can move forward with probability $p$ [resp. $q$], or with probability $1-p$ [resp. $1-q$] it remains at its present position. For the case $q>0$, several limit theorems are obtained by Coletti, Gava, and de Lima (2019). In this paper we prove limit theorems for the case $q=0$, where the walker can exhibit all three forms of asymptotic behavior as $p$ is varied. As a byproduct, we obtain limit theorems for the cluster size of the root in percolation on uniform random recursive trees.
\end{abstract}

\maketitle

\section{Introduction}
\label{intro}

The elephant random walk, introduced by Sch\"{u}tz and Trimper \cite{SchutzTrimper04}, is defined as follows: \begin{itemize}
\item The first step $Y_1$ of the walker is $+1$ with probability $s$, and $-1$ with probability $1-s$.
\item For each $n=1,2,\cdots$, let $U_n$ be uniformly distributed on $\{1,\cdots,n\}$, and
\begin{align*}
Y_{n+1} &=  \begin{cases}
Y_{U_n} &\mbox{with probability $p$}, \\
-Y_{U_n} &\mbox{with probability $1-p$}. \\
\end{cases}
\end{align*} 
\end{itemize}
Each of choices in the above procedure is made independently. The sequence $\{Y_i\}$ generates a one-dimensional random walk $\{S_n\}$ by
\[ S_0:=0,\quad \mbox{and} \quad S_n:= \sum_{i=1}^n Y_i \quad \mbox{for $n=1,2,\cdots$.} \]
It admits a phase transition from diffusive to superdiffusive behavior at the critical value $p_c=3/4$. Several limit theorems are obtained in \cite{BaurBertoin16,Bercu18,Collettietal17a,Collettietal17b,KubotaTakei19JSP}. Variations of elephant random walks studied mainly from mathematical viewpoint are found in \cite{BercuLaulin19,Bertoin18,Bertoin19,Businger18,GutStadtmuller18,GutStadtmuller19}.

Kumar, Harbola, and Lindenberg \cite{KumarHarbolaLindenberg10PRE} proposed a random walk model of elephant type, which exhibits asymptotic subdiffusion, normal diffusion, and superdiffusion as a single parameter is swept. An even simpler model of this kind was introduced by Harbola, Kumar, and Lindenberg \cite{HarbolaKumarLindenberg14PRE}. Assume that $p \in (0,1)$, $q \in [0,1)$ and $s \in [0,1]$. Define a sequence $\{X_i\}$ of $\{0,1\}$-valued random variables as follows:
 \begin{itemize}
\item $P(X_1=1)=1-P(X_1=0)=s$.
\item For each $n=1,2,\cdots$, let $U_n$ be uniformly distributed on $\{1,\cdots,n\}$. 
\begin{itemize}
\item[$\triangleright$] If $X_{U_n}=1$, then $X_{n+1} =  \begin{cases}
1 &\mbox{with probability $p$}, \\
0 &\mbox{with probability $1-p$}. \\
\end{cases}$
\item[$\triangleright$] If $X_{U_n}=0$, then $X_{n+1} =  \begin{cases}
1 &\mbox{with probability $q$}, \\
0 &\mbox{with probability $1-q$}. \\
\end{cases}$
\end{itemize}
\end{itemize}
A random walk $\{H_n\}$ on $\mathbb{Z}_+ := \{0,1,2,\cdots\}$ is defined by
\[ H_0:=0,\quad \mbox{and} \quad  H_n:=\sum_{i=1}^n X_i \quad \mbox{for $n=1,2,\cdots$.} \]
Note that for $n=1,2,\cdots$, the conditional distribution of $X_{n+1}$ given the history up to time $n$ is
\begin{align*} 
&P(X_{n+1} = 1 \mid X_1,\cdots,X_n) = 1- P(X_{n+1} = 0 \mid X_1,\cdots,X_n)\\
&= p \cdot \dfrac{\#\{i=1,\cdots,n : X_i=1\}}{n} + q \cdot \dfrac{\#\{i=1,\cdots,n : X_i=0\}}{n}. 
\end{align*}
Since
\begin{align*}
\#\{i=1,\cdots,n : X_i=1\} = H_n \quad \mbox{and} \quad \#\{i=1,\cdots,n : X_i=0\} = n-H_n,
\end{align*}
the conditional expectation of $X_{n+1}$ is
\begin{align}
E[X_{n+1} \mid X_1,\cdots,X_n] = P(X_{n+1} = 1 \mid X_1,\cdots,X_n) &=\alpha \cdot \dfrac{H_n}{n} + q,
\label{eq:HarbolaKumarLindenberg14PREtrans}
\end{align}
where $n=1,2,\cdots$ and $\alpha:=p-q \in (-1,1)$. Solving
\begin{align*}
E[H_{n+1}] = \left(1+\dfrac{\alpha}{n} \right) E[H_n] + q,
\end{align*}
we have
\begin{align}
E[H_n] &= \dfrac{qn}{1-\alpha} + \left(s -\dfrac{q}{1-\alpha} \right)\cdot \dfrac{\Gamma(n+\alpha)}{\Gamma(1+\alpha) \Gamma(n)} \label{eq:CGL19(3)} \\
&\sim \begin{cases}
\dfrac{qn}{1-\alpha}&(q>0), \\[2mm]
 \dfrac{s}{\Gamma(1+p)} \cdot n^p &(q=0),
\end{cases} \notag
\end{align}
where $x_n \sim y_n$ means that $x_n/y_n$ converges to $1$ as $n \to \infty$. The walker is ballistic if $q>0$. On the other hand, in the case $q=0$, which we call the {\it `laziest' minimal random walk model of elephant type}, all three phases of asymptotic behavior are observed as $p \in (0,1)$ varies (see \cite{HarbolaKumarLindenberg14PRE}).

Coletti, Gava, and de Lima \cite{CollettiGavadeLima19JSM} proved several limit theorems for this model. The strong law of large numbers holds for any $q \in [0,1)$:
\[ \lim_{n \to \infty} \dfrac{H_n-E[H_n]}{n}= 0 \qquad \mbox{a.s..} \]
In particular, together with \eqref{eq:CGL19(3)},
\begin{align}
 \lim_{n \to \infty} \dfrac{H_n}{n}= \dfrac{q}{1-\alpha}\qquad \mbox{a.s..}
 \label{eq:Colettietal19sLLN}
\end{align}
For $q>0$, further limit theorems are proved in \cite{CollettiGavadeLima19JSM}.
Let $\rho := q/(1-\alpha) \in (0,1)$ and $\phi(t):=\sqrt{2t\log\log t}$.
\begin{itemize}
\item[(i)] If $0 \leq \alpha <1/2$, then
\begin{align*}
\dfrac{H_n-E[H_n]}{\sqrt{\tfrac{\rho(1-\rho)}{1-2\alpha}n}} \stackrel{d}{\to} N\left(0,1\right),\mbox{ and } 
\limsup_{n \to \infty} \pm \dfrac{H_n-E[H_n]}{\phi\left(\tfrac{\rho(1-\rho)}{1-2\alpha}n\right)}=1 
\mbox{ a.s..}
\end{align*}
\item[(ii)] If $\alpha=1/2$, then
\begin{align*}
\dfrac{H_n-E[H_n]}{\sqrt{\rho(1-\rho)n\log n}} \stackrel{d}{\to} N\left(0,1\right), \mbox{ and } 
\limsup_{n \to \infty} \pm \dfrac{H_n-E[H_n]}{\phi\left(\rho(1-\rho)n\log n\right)}=1
\mbox{ a.s..}
\end{align*}
\item[(iii)] If $1/2<\alpha<1$, then there exists a random variable $W$ with positive variance such that 
\begin{align*}
\displaystyle \lim_{n \to \infty} \dfrac{H_n-E[H_n]}{a_n} = W \quad \mbox{a.s. and in $L^2$,}
\end{align*}
where $a_n:=\dfrac{\Gamma(n+\alpha)}{\Gamma(1+\alpha) \Gamma(n)}$. 
Moreover, essentially the same calculation as in \cite{KubotaTakei19JSP} gives that
\begin{align*}
\dfrac{H_n-E[H_n]-W \cdot a_n}{\sqrt{\tfrac{\rho(1-\rho)}{2\alpha-1}n}} \stackrel{d}{\to} N\left(0,1\right),
\intertext{ and }
\limsup_{n \to \infty} \pm \dfrac{H_n-E[H_n]-W \cdot a_n}{\phi \left( \tfrac{\rho(1-\rho)}{2\alpha-1}n \right)}
=1 
\mbox{ a.s..}
\end{align*}
\end{itemize}

\begin{remark} If $s=\rho \in (0,1)$, then the minimal random walk model is equivalent to the ``correlated Bernoulli process" introduced by Drezner and Farnum \cite{DreznerFarnum93}. In the latter context, the central limit theorem was proved by Heyde \cite{Heyde04}.
\end{remark}

Although the $q=0$ case is the most interesting, only partial results are obtained in Coletti, Gava, and de Lima \cite{CollettiGavadeLima19JSM}. 
The aim of this article is to prove several limit theorems for this case.

\section{Results} \label{sec:Results}

In this section we consider the `laziest' minimal random walk model of elephant type ($q=0$). Note that if $X_1=0$, then $X_n=0$ for all $n$. Hereafter we assume that $s=1$ in addition. The dynamics can be summarized as follows: Let $p \in (0,1)$.
\begin{itemize}
\item The first step $X_1$ of the walker is $1$ with probability one.
\item For each $n=1,2,\cdots$, let $U_n$ be uniformly distributed on $\{1,\cdots,n\}$, and
\begin{align*}
X_{n+1} &=  \begin{cases}
X_{U_n} &\mbox{with probability $p$}, \\
0 &\mbox{with probability $1-p$}. \\
\end{cases}
\end{align*} 
\end{itemize}

The equation \eqref{eq:HarbolaKumarLindenberg14PREtrans} becomes
\begin{align}
 E[X_{n+1} \mid \mathcal{F}_n]=P(X_{n+1}=  1 \mid \mathcal{F}_n) 
 = p \cdot \dfrac{H_n}{n},  \label{eq:elephantRWCondDistp}
\end{align}
where $\mathcal{F}_n$ is the $\sigma$-algebra generated by $X_1,\cdots,X_n$.
Noting that
\[ E[H_{n+1} \mid \mathcal{F}_n]=\left(1+\dfrac{p}{n}\right) H_n, \]
we introduce 
\begin{align} \label{eq:ElephantRWa_nDef}
 a_0:=1, \quad \mbox{and} \quad 
 a_n := \prod_{k=1}^{n-1} \left(1+\dfrac{p}{k}\right)= \dfrac{\Gamma(n+p)}{\Gamma(n)\Gamma(1+p)}\quad \mbox{for $n=1,2,\cdots$,}
\end{align}
and set
\begin{align*}
\widehat{M_n}:=\dfrac{H_n}{a_n}\quad \mbox{for $n=0,1,2,\cdots$.}
\end{align*}
Then $\{ \widehat{M_n} \}$ satisfies a martingale property $E[\widehat{M_{n+1}} \mid \mathcal{F}_n]=\widehat{M_n}$. 
Since $ \widehat{M_n} $ is nonnegative, Doob's convergence theorem implies that
\begin{align}
\lim_{n \to \infty} M_n = \lim_{n \to \infty}\dfrac{H_n}{a_n} = \widehat{W} \quad \mbox{a.s..} 
\label{eq:LaziestMartLim1}
\end{align}
In Corollary \ref{cor:MiyazakiTakeiMoments} below, we show that $\widehat{W}$ is positive with probability one.

\subsection{Moments of the position}

For $k=1,2,\cdots$, let
\[ a_n^{(k)} := \dfrac{\Gamma(n+kp)}{\Gamma(n)\Gamma(1+kp)}. \]
Note that $a_n^{(1)}=a_n$. The moments of the position $H_n$ up to the fourth are calculated in section 4.6 of Coletti, Gava, and de Lima \cite{CollettiGavadeLima19JSM}:
\begin{align} 
\begin{aligned}
E[H_n] &= a_n^{(1)}, \\
E[(H_n)^2] &= 2a_n^{(2)}-a_n^{(1)}, \\
E[(H_n)^3] &= 6a_n^{(3)}-6a_n^{(2)}+a_n^{(1)}, \\
E[(H_n)^4] &= 24a_n^{(4)}-36a_n^{(3)}+14a_n^{(2)}-a_n^{(1)}.
\end{aligned} 
\label{eq:ColettideLimaGavaSec4.6}
\end{align}
We could not find a simple way to describe the coefficients.
Let $(x)_1:=x$ and $(x)_k:=x(x-1)\cdots(x-k+1)$ for $k=2,3,\cdots$.
The $k$-th factorial moment of a random variable $X$ is defined by $E[(X)_k]$.
Using \eqref{eq:ColettideLimaGavaSec4.6}, we can see that
\begin{align*}
E[(H_n)_2] &= E[(H_n)^2]-E[H_n] = 2( a_n^{(2)}-a_n^{(1)}), \\
E[(H_n)_3] &= E[(H_n)^3]-3E[(H_n)^2]+2E[H_n] = 6(a_n^{(3)}-2a_n^{(2)}+a_n^{(1)}),\\
E[(H_n)_4] &= E[(H_n)^4]-6E[(H_n)^3]+11E[(H_n)^2]-6E[H_n] \\
&= 24(a_n^{(4)}-3a_n^{(3)}+3a_n^{(2)}-a_n^{(1)}).
\end{align*}

The following theorem, which will be proved in section \ref{sec:Moments}, gives the general solution for $E[(H_n)_k]$.

\begin{theorem} \label{thm:MiyazakiTakeiMoments} Assume that $p \in (0,1)$, $q=0$, and $s=1$. For any $k=1,2,\cdots$ and $n=1,2,\cdots$, 
\begin{align}
E[(H_n)_k] = k! \cdot \sum_{i=1}^k (-1)^{k-i} \binom{k-1}{i-1} a_n^{(i)}. \label{eq:MiyazakiTakeiMoments}
\end{align}
\end{theorem}

Since $a_n^{(k)} \sim \dfrac{n^{kp}}{\Gamma(1+kp)}$ as $n \to \infty$, 
we have
\[ \lim_{n \to \infty} \dfrac{E[(H_n)^k]}{a_n^{(k)}} = \lim_{n \to \infty} \dfrac{E[(H_n)_k]}{a_n^{(k)}} = k!
\quad \mbox{for any $k=1,2,\cdots$}. \]
The following corollary is a much more precise result than \eqref{eq:Colettietal19sLLN} for $q=0$.

\begin{corollary} \label{cor:MiyazakiTakeiMoments}
For any $k=1,2,\cdots$, 
\[ \lim_{n \to \infty} E \left[\left(\dfrac{H_n}{n^p}\right)^k\right] = \dfrac{k!}{\Gamma(1+kp)}.  \]
Thus the martingale $\{\widehat{M}_n\}$ is $L^k$-bounded for any $k$, and
the almost sure limit 
\begin{align}
\mathcal{W}:= \lim_{n \to \infty}\dfrac{H_n}{n^p} = \dfrac{\widehat{W}}{\Gamma(1+p)} 
\label{eq:LaziestMartLim2}
\end{align}
has a Mittag--Leffler distribution with parameter $p$ (see Appendix \ref{app:MLdistr}). In particular, $P(\mathcal{W}>0)=1$.
\end{corollary}



\subsection{Limit theorems for the laziest case} 

By Corollary \ref{cor:MiyazakiTakeiMoments}, the limit $\widehat{W}$ in \eqref{eq:LaziestMartLim1} satisfies $P(\widehat{W}>0)=1$. 
Based on this fact, we obtain central limit theorems in the following form.

\begin{theorem} \label{thm:MiyazakiTakeiCLT} Assume that $q=0$ and $s=1$. For $0<p<1$, 
\begin{align*}
\dfrac{H_n-\widehat{W} \cdot a_n}{\sqrt{\widehat{W} \cdot a_n}} \stackrel{d}{\to} N\left(0,1\right), 
\intertext{and}
\dfrac{H_n-\widehat{W} \cdot a_n}{\sqrt{n^p}} \stackrel{d}{\to} \sqrt{\mathcal{W}'} \cdot Z,
\end{align*}
where $Z$ is distributed as $N(0,1)$, and $ \mathcal{W}'$ is independent of $Z$ and has the same distribution as $ \mathcal{W}$.
\end{theorem}

This situation is quite different from $q>0$: The central limit theorem holds for whole regions of parameter space, with random centering and random norming.

We also prove the law of the iterated logarithm.

\begin{theorem} \label{thm:MiyazakiTakeiLIL} Assume that $q=0$ and $s=1$. For $0<p<1$, 
\begin{align*}
\limsup_{n \to \infty} \pm \dfrac{H_n-\widehat{W} \cdot a_n}{\phi \left( \widehat{W} \cdot a_n \right)}
=1 
\mbox{ a.s.,}
\end{align*}
where $\phi(t):=\sqrt{2t\log \log t}$. 
\end{theorem}

Theorems \ref{thm:MiyazakiTakeiCLT} and \ref{thm:MiyazakiTakeiLIL} will be proved in section \ref{sec:ProofLimitThm}.

\subsection{Applications to percolation on random recursive trees} \label{sec:ApplPercRRT}

K\"{u}rsten \cite{Kursten16} found important connections between (several variations of) elephant random walks and percolation on random recursive trees.

Consider the following procedure for obtaining a sequence $\{T_i\}$ of recursive trees: 
The first graph $T_1$ consists of a single vertex labeled 1. For each $i=2,3,\cdots$, the graph $T_i$ is evolved from $T_{i-1}$ by joining a new vertex labeled $i$ to a uniformly chosen vertex labeled $u_{i-1}$ from $T_{i-1}$.
We perform Bernoulli bond percolation on $T_n$: Each edge of $T_n$ is independently removed with probability $1-p$, and otherwise retained. Then we can see that the size $\# \mathcal{C}_{1,n}$ of the cluster containing the vertex labeled 1 has the same distribution as the position $H_n$ of the laziest minimal random walk model of elephant type (this relation is implicitly mentioned before eq. (43) in \cite{Kursten16}).

Thus all of our results described above have counterparts for $\# \mathcal{C}_{1,n}$. To our best knowledge, Theorems \ref{thm:MiyazakiTakeiMoments}, \ref{thm:MiyazakiTakeiCLT} and \ref{thm:MiyazakiTakeiLIL} are new also in this context. We remark that the $\# \mathcal{C}_{1,n}$-version of Corollary \ref{cor:MiyazakiTakeiMoments} is Lemma 3 in Businger \cite{Businger18}, where it plays a crucial role in analyzing the shark random swim, and is proved using a connection with the Yule process. 
Our proof of Corollary \ref{cor:MiyazakiTakeiMoments} based on Theorem \ref{thm:MiyazakiTakeiMoments} is a short alternative.

In Appendix \ref{sec:minRWandPerc}, we give a precise description of the relation between the minimal random walk model with $0\leq q < p <1$ and percolation on random recursive trees. As an easy and useful application of this, we derive new expressions of the expectation and the variance of $H_n$, in terms of size of percolation clusters.

\section{Factorial moments of the position} \label{sec:Moments}

To prove Theorem \ref{thm:MiyazakiTakeiMoments},
we use the probability generating functions: Let
\[ f_n(x) := E[x^{H_n}] \qquad \mbox{for $n=1,2,\cdots$.} \]
Recall that $f_n^{(k)}(1)=E[(H_n)_k]$, where $f_n^{(k)}(1)$ denotes the $k$-th derivative of $f_n(x)$ at $x=1$. By \eqref{eq:elephantRWCondDistp}, we can see
\begin{align*}
 E[x^{H_{n+1}} \mid \mathcal{F}_n] &= x^{H_n} \cdot \left\{ x^1 \cdot p \cdot \dfrac{H_n}{n} +  x^0 \cdot \left(1-p \cdot \dfrac{H_n}{n}\right)\right\} \\
 &= x^{H_n} \cdot \left\{ \dfrac{p(x-1)}{n} \cdot H_n +1\right \},
\end{align*}
and
\begin{align*}
f_{n+1}(x) &= E\left[ x^{H_n} \cdot \left\{ \dfrac{p(x-1)}{n} \cdot H_n +1 \right\} \right] \\
&= \dfrac{p(x-1)}{n} \cdot E[H_n x^{H_n} ] + E[x^{H_n}] 
= \dfrac{px(x-1)}{n} \cdot f'_n(x) + f_n(x).
\end{align*}
The Leibniz rule yields 
\begin{align*}
f_{n+1}^{(k)}(x) &= \dfrac{px(x-1)}{n} \cdot f_n^{(k+1)}(x) + \binom{k}{1} \cdot \dfrac{2px-p}{n} \cdot f_n^{(k)}(x)\\
&\quad + \binom{k}{2} \cdot \dfrac{2p}{n} \cdot f_n^{(k-1)}(x)+f_n^{(k)}(x).
\end{align*}
Thus we have
\begin{align*}
E[(H_{n+1})_k] = f_{n+1}^{(k)}(1) 
&= \dfrac{kp}{n} \cdot f_n^{(k)}(1)+ \dfrac{k(k-1)p}{n} \cdot f_n^{(k-1)}(1)+f_n^{(k)}(1) \\
&= \dfrac{n+kp}{n} \cdot E[(H_n)_k] + \dfrac{k(k-1)p}{n} \cdot E[(H_n)_{k-1}].
\end{align*}

We prove \eqref{eq:MiyazakiTakeiMoments} by induction. For each $n=1,2,\cdots$, $E[H_n] = a_n^{(1)}$ in \eqref{eq:ColettideLimaGavaSec4.6} clearly satisfies \eqref{eq:MiyazakiTakeiMoments}. On the other hand,
\[ E[H_1]=1,\quad \mbox{and} \quad E[(H_1)_k] = 0 \quad \mbox{for $k=2,3,\cdots$} \]
satisfy \eqref{eq:MiyazakiTakeiMoments}, as $a_1^{(1)}=1$.
Assume that $k>2$, and that $E[(H_n)_{k-1}]$ and $E[(H_n)_k]$ satisfy \eqref{eq:MiyazakiTakeiMoments}. We can see that
\begin{align*}
E[(H_{n+1})_k]&= \dfrac{n+kp}{n} \cdot k! \cdot \sum_{i=1}^k (-1)^{k-i} \binom{k-1}{i-1} a_n^{(i)} \\
&\quad + \dfrac{k(k-1)p}{n} \cdot (k-1)! \cdot \sum_{i=1}^{k-1} (-1)^{k-1-i} \binom{k-2}{i-1} a_n^{(i)} \\
&= \dfrac{n+kp}{n} \cdot k! \cdot  a_n^{(k)} \\
&\quad + k! \cdot \sum_{i=1}^{k-1} (-1)^{k-i} \binom{k-1}{i-1} a_n^{(i)} \left(\dfrac{n+kp}{n}-\dfrac{(k-1)p}{n} \cdot \dfrac{k-i}{k-1} \right) \\
&= k! \cdot  a_{n+1}^{(k)} + k! \cdot \sum_{i=1}^{k-1} (-1)^{k-i} \binom{k-1}{i-1} a_n^{(i)} \cdot \dfrac{n+ip}{n} \\
&= k! \cdot \sum_{i=1}^k (-1)^{k-i} \binom{k-1}{i-1} a_{n+1}^{(i)}.
\end{align*}
This completes the proof of Theorem \ref{thm:MiyazakiTakeiMoments}.

\section{Limit theorems} \label{sec:ProofLimitThm}

The structure of our proof of Theorems \ref{thm:MiyazakiTakeiCLT} and \ref{thm:MiyazakiTakeiLIL} is similar to that of \cite{KubotaTakei19JSP}.  
For $n=0,1,2,\cdots$, we set
\[ M_n:=\dfrac{H_n-E[H_n]}{a_n}(=\widehat{M_n}-1), \]
where $a_n$ is defined in \eqref{eq:ElephantRWa_nDef}. Clearly $\{M_n \}$ is a martingale with mean zero.

As in the proof of Lemma 1 of \cite{KubotaTakei19JSP},
we can obtain
\begin{align} 
M_{n+1}-M_n 
= \dfrac{X_{n+1}-E[X_{n+1} \mid \mathcal{F}_n]}{a_{n+1}} \label{eq:ElephantRWmartdiff1}
\end{align}
for each $n=1,2,\cdots$.
Noting that $X_n^2=X_n$, we have
\begin{align}
E[(M_{n+1}-M_n)^2 \mid \mathcal{F}_n] 
&=\dfrac{E[(X_{n+1}-E[X_{n+1} \mid \mathcal{F}_n])^2 \mid \mathcal{F}_n]}{(a_{n+1})^2} \notag \\
&=\dfrac{E[X_{n+1}^2 \mid \mathcal{F}_n] -(E[X_{n+1} \mid \mathcal{F}_n])^2}{(a_{n+1})^2} \notag \\
&=\dfrac{E[X_{n+1} \mid \mathcal{F}_n] \cdot (1 -E[X_{n+1} \mid \mathcal{F}_n])}{(a_{n+1})^2}.\label{eq:ElephantRWmartdiff2}
\end{align}
Note that \eqref{eq:ElephantRWmartdiff1} and \eqref{eq:ElephantRWmartdiff2} hold also for $n=0$, where $\mathcal{F}_0$ is the trivial $\sigma$-algebra.


For $k=1,2,\cdots$, let
\[ d_k:=M_k-M_{k-1}=\dfrac{X_k-E[X_k \mid \mathcal{F}_{k-1}]}{a_k}. \]
Note that $|d_k| \leq \dfrac{1}{a_k} \leq 1$.
Using \eqref{eq:ElephantRWmartdiff2}, \eqref{eq:Colettietal19sLLN}, and \eqref{eq:LaziestMartLim1},
\begin{align*}
E\left[ (d_k)^2 \mid \mathcal{F}_{k-1} \right] &= \dfrac{E[X_k \mid \mathcal{F}_{k-1}] \cdot (1 -E[X_k \mid \mathcal{F}_{k-1}])}{(a_k)^2} \\
&= p \cdot \dfrac{H_{k-1}}{k-1} \cdot \left(1-p \cdot \dfrac{H_{k-1}}{k-1}\right) \cdot \dfrac{1}{(a_k)^2} \\
&= p \cdot \dfrac{H_{k-1}}{a_{k-1}} \cdot \dfrac{a_{k-1}}{k-1} \cdot \left(1-p \cdot \dfrac{H_{k-1}}{k-1}\right) \cdot \dfrac{1}{(a_k)^2} \\
&\sim \dfrac{p \cdot \widehat{W}}{ka_k} \qquad \mbox{as $k \to \infty$ a.s..}
\end{align*}
As $|d_k| \leq 1$, the bounded convergence theorem yields that
\begin{align*}
E[ (d_k)^2 ] \sim \dfrac{p \cdot E[\widehat{W}]}{ka_k} =\dfrac{p}{ka_k} \qquad \mbox{as $k \to \infty$.}
\end{align*}
Thus we have
\begin{align*}
 V_n^2&:= \sum_{k=n}^{\infty} E\left[ (d_k)^2 \mid \mathcal{F}_{k-1} \right] \\
 &\sim p \cdot \widehat{W} \cdot \Gamma(1+p) \sum_{k=n}^{\infty} \dfrac{1}{k^{1+p}} \\
 &\sim p \cdot \widehat{W} \cdot \Gamma(1+p) \cdot \dfrac{1}{pn^p}
 =\frac{\widehat{W} \cdot \Gamma(1+p)}{n^p} \sim \dfrac{\widehat{W}}{a_n}\qquad \mbox{as $n  \to \infty$ a.s.,} 
\end{align*}
and
\begin{align*}
 s_n^2&:=\sum_{k=n}^{\infty} E[(d_k)^2] \sim \dfrac{1}{a_n}\qquad \mbox{as $n  \to \infty$,} 
\end{align*}
which imply that 
\begin{align}
\lim_{n \to \infty} \dfrac{V_n^2}{s_n^2}=\widehat{W} \quad \mbox{a.s..} \label{eq:KubotaTakeiSupercriticalCLTLILV_nLLN}
\end{align}

\begin{proof}[Proof of Theorems \ref{thm:MiyazakiTakeiCLT} and \ref{thm:MiyazakiTakeiLIL}] 
We check the conditions of Theorem \ref{thm:Heyde77Theorem1b} (ii) and (iii) in Appendix \ref{sec:MartLimitThm} are satisfied.

To prove conditions a) and a') 
with $\eta^2=\widehat{W}$ are satisfied, we will show
\[
\lim_{n \to \infty} \dfrac{1}{s_n^2}\sum_{k=n}^{\infty} \{ (d_k)^2 - E[(d_k)^2 \mid \mathcal{F}_{k-1}]\} = 0 \quad \mbox{a.s..}
\]
By the tail version of Kronecker's lemma (see Lemma 1 (ii) in Heyde \cite{Heyde77}), it is sufficient to show
\begin{align}
\sum_{k=1}^{\infty} \dfrac{1}{s_k^2} \{ (d_k)^2 - E[(d_k)^2 \mid \mathcal{F}_{k-1}]\}<+\infty \quad \mbox{a.s..} \label{eq:WnVnKroneckerTail}
\end{align}
Let $\widetilde{d}_k$ denote the summand. Note that
\[
\sum_{k=1}^{\infty} E[ (\widetilde{d}_k)^2 \mid \mathcal{F}_{k-1}] \leq \sum_{k=1}^{\infty} \dfrac{1}{s_k^4} E[(d_k)^4 \mid \mathcal{F}_{k-1}].
\]
Since
\begin{align*}
E[(d_k)^4 \mid \mathcal{F}_{k-1}] 
&\leq \dfrac{1}{(a_k)^2} \cdot E[(d_k)^2 \mid \mathcal{F}_{k-1}] \\
&\sim s_k^4 \cdot \dfrac{p \cdot \widehat{W}}{ka_k} \sim s_k^4 \cdot \dfrac{p \cdot \widehat{W} \cdot \Gamma(1+p)}{k^{1+p}}\quad \mbox{as $k  \to \infty$,} 
\end{align*}
the series in the right hand side converges a.s..
Theorem \ref{thm:HallHeyde80Thm215} implies \eqref{eq:WnVnKroneckerTail}.

For $\varepsilon>0$, noting that
\[
E[(d_k)^2 : |d_k| > \varepsilon s_n] \leq \dfrac{1}{\varepsilon^2 s_n^2} E[(d_k)^4],
\]
and
\[
E[(d_k)^4] \leq \dfrac{1}{(a_k)^2} \cdot E[(d_k)^2] \sim \dfrac{p  \cdot \Gamma(1+p)^3}{k^{1+3p}}\quad \mbox{as $k  \to \infty$,}
\]
we have
\begin{align*}
\dfrac{1}{s_n^2} \sum_{k=n}^{\infty} E[(d_k)^2 : |d_k| > \varepsilon s_n] &\leq \dfrac{1}{\varepsilon^2 s_n^4} \sum_{k=n}^{\infty} E[(d_k)^4] \\
&\leq C_1 n^{2p} \cdot \dfrac{1}{n^{3p}} = \dfrac{C_1}{n^p} \to 0\quad \mbox{as $n \to \infty$}.
\end{align*}
In view of Remark \ref{rem:Heyde77Theorem1b}, condition b) is also satisfied.

Similarly, for $\varepsilon>0$ we have
\begin{align*}
\dfrac{1}{s_k}E[|d_k| : |d_k| > \varepsilon s_k] &\leq \dfrac{1}{s_k} \cdot \dfrac{1}{\varepsilon^3 s_k^3} E[(d_k)^4] \\
&\leq C_2 k^{2p} \cdot \dfrac{1}{k^{1+3p}} = \dfrac{C_2}{k^{1+p}},
\end{align*}
which implies that condition c) holds.

Condition d) is implied by
\[
\displaystyle \sum_{n=1}^{\infty} \dfrac{1}{s_n^4} E[(d_n)^4]<+\infty.
\]

To deduce the conclusions of Theorems \ref{thm:MiyazakiTakeiCLT} and \ref{thm:MiyazakiTakeiLIL} from Theorem \ref{thm:Heyde77Theorem1b} (ii) and (iii), note that
\begin{align*}
 W-M_n = \dfrac{a_n \cdot W - (H_n-E[H_n])}{a_n} = \dfrac{a_n \cdot \widehat{W} - H_n}{a_n},
\intertext{and}
 a_n \cdot \widehat{\phi}\left(\dfrac{\widehat{W}}{a_n} \right) \sim  \widehat{\phi}\left( a_n \cdot \widehat{W} \right)\quad \mbox{as $n \to \infty$},
\end{align*}
where $\widehat{\phi}(t)=\sqrt{2t\log |\log t|}$.
\end{proof}

\appendix

\section{The minimal random walk model and percolation on random recursive trees} \label{sec:minRWandPerc}

We explore a relation between the minimal random walk model by Harbola, Kumar, and Lindenberg \cite{HarbolaKumarLindenberg14PRE}, explained in the Introduction, and percolation on random recursive trees. Throughout this section we assume that $0 \leq q < p < 1$, and set $\alpha=p-q$ and $\rho = q/(1-\alpha)$. 

The sequence $\{T_i\}$ of random recursive trees is defined in section \ref{sec:ApplPercRRT}. 
Consider bond percolation on $T_n$ with parameter $\alpha$.
The expectation regarding this model is denoted by $E_{\alpha} [ \,\cdot\,]$.
There are at most $n$ clusters, which are denoted by $\mathcal{C}_{1,n},\,\mathcal{C}_{2,n},\,\cdots,\,\mathcal{C}_{n,n}$ (for convenience we regard $\mathcal{C}_{j,n}=\emptyset$ if $j$ is larger than the number of clusters).
We quote some of results in K\"{u}rsten \cite{Kursten16}.

\begin{lemma} \label{lem:Kursten16summary} For bond percolation on $T_n$ with parameter $\alpha \in (0,1)$,
\begin{align}
E_{\alpha} [\# \mathcal{C}_{1,n} ] &= \dfrac{\Gamma(n+\alpha)}{\Gamma(1+\alpha)\Gamma(n)}, \label{eq:Kuersten16PRE(13)} \\
E_{\alpha} [(\# \mathcal{C}_{1,n})^2 ] &= \dfrac{2\Gamma(n+2\alpha)}{\Gamma(1+2\alpha)\Gamma(n)} - \dfrac{\Gamma(n+\alpha)}{\Gamma(1+\alpha)\Gamma(n)}, \label{eq:Kuersten16PRE(44)} \\
\sum_{j=1}^n E_{\alpha} [(\# \mathcal{C}_{j,n})^2 ] &= \begin{cases}
\dfrac{1}{1-2\alpha} \cdot n+ \dfrac{1}{2\alpha-1} \cdot \dfrac{\Gamma(n+2\alpha)}{\Gamma(2\alpha)\Gamma(n)}  &(\alpha \neq 1/2), \\[3mm]
\displaystyle n\sum_{\ell=1}^n \dfrac{1}{\ell}&(\alpha=1/2).
\end{cases} \label{eq:Kuersten16PRE(17)}
\end{align}
\end{lemma}

\begin{remark} In \cite{Kursten16}, \eqref{eq:Kuersten16PRE(13)} and \eqref{eq:Kuersten16PRE(17)} are derived from basic results on the original elephant random walk, found in \cite{SchutzTrimper04}. In view of the connection with the laziest minimal random walk model explained in section \ref{sec:ApplPercRRT}, \eqref{eq:Kuersten16PRE(13)} and \eqref{eq:Kuersten16PRE(44)} are paraphrases of \eqref{eq:ColettideLimaGavaSec4.6}. Those are obtained by solving relatively easy difference equations.
\end{remark}

Let $\xi_1,\xi_2,\cdots,\xi_n$ be a sequence of independent random variables, which is also independent from bond percolation, satisfying
\begin{align*}
&P_{p,q,s} (\xi_1=1)=1-P_{p,q,s}(\xi_1=0)=s,\quad \mbox{and} \\
&P_{p,q,s}(\xi_j=1)=1-P_{p,q,s}(\xi_j=0)=\rho \quad \mbox{for $j>1$.}
\end{align*}
The transition probability 
in \eqref{eq:HarbolaKumarLindenberg14PREtrans} can be interpreted as follows: For each step $n=2,3,\cdots$,
\begin{itemize}
\item with probability $\alpha$, the walker repeats the behavior at a uniformly chosen time, and
\item with probability $1-\alpha$, the walker moves forward with probability $\rho$, or remains at rest otherwise.
\end{itemize}
Similarly to \cite{Kursten16}, we can see that $H_n$ has the same distribution as
\begin{align}
\sum_{j=1}^n \xi_j \cdot (\# \mathcal{C}_{j,n}). \label{eq:Kuersten16(11)bis}
\end{align}

In the case $q>0$, computation of the second moment (and the variance) of $H_n$ by solving difference equations is straightforward but quite tedious, as is imagined from very complicated equations (8), (9) and (10) in \cite{HarbolaKumarLindenberg14PRE}. Using the above connection with percolation, we can easily obtain concise formulae described in terms of the moments of the size of open clusters. 

\begin{theorem} \label{thm:MiyazakiTakei20AppB} Let $E_{p,q,s}$ and $V_{p,q,s}$ denote the expectation and the variance for the minimal random walk model. Assume that $0 \leq q < p < 1$, and set $\alpha=p-q$ and $\rho = q/(1-\alpha)$. Then we have the following.
\begin{align}
E_{p,q,s} [H_n] &= \rho n + (s-\rho) E_{\alpha} [\# \mathcal{C}_{1,n} ], \label{eq:MiyazakiTakei20AppB-E} \\
V_{p,q,s} [H_n] 
&= \rho (1-\rho) \sum_{j=1}^n E_{\alpha}[(\# \mathcal{C}_{j,n})^2] \notag \\
&\quad  + (1-2\rho)(s-\rho) E_{\alpha} [(\# \mathcal{C}_{1,n})^2 ] - (s-\rho)^2 (E_{\alpha} [\# \mathcal{C}_{1,n}])^2.  \label{eq:MiyazakiTakei20AppB-V}
\end{align}
\end{theorem}

\begin{proof} Since $E_{p,q,s} [\xi_1]=s$ and $E_{p,q,s} [\xi_j]=\rho$ for $j>1$, we have
\begin{align*}
E_{p,q,s} [H_n] &= \sum_{j=1}^n E_{p,q,s} [\xi_j] \cdot E_{\alpha} [\# \mathcal{C}_{j,n}] 
= s E_{\alpha}[\# \mathcal{C}_{1,n}]+\rho \sum_{j=2}^n E_{\alpha}[\# \mathcal{C}_{j,n}] \\
&= \rho \sum_{j=1}^n E_{\alpha}[\# \mathcal{C}_{j,n}] + (s-\rho) E_{\alpha} [\# \mathcal{C}_{1,n}]
= \rho n + (s-\rho) E_{\alpha}[\# \mathcal{C}_{1,n}].
\end{align*}
Turning to the mean square displacement, similarly as eq. (17) in \cite{Kursten16},
\begin{align} 
&E_{p,q,s}[(H_n)^2]\notag \\
&= \sum_{j=1}^n E_{p,q,s} [(\xi_j)^2] \cdot E_{\alpha} [(\# \mathcal{C}_{j,n})^2] \notag\\
&\quad + 2\sum_{1\leq j<k \leq n} E_{p,q,s}[\xi_j] \cdot E_{p,q,s}[\xi_k] \cdot E_{\alpha} [(\# \mathcal{C}_{j,n}) \cdot (\# \mathcal{C}_{k,n})] \notag \\
&= \rho \sum_{j=1}^n E_{\alpha}[(\# \mathcal{C}_{j,n})^2] + 2\rho^2 \sum_{1 \leq j<k \leq n} E_{\alpha}[(\# \mathcal{C}_{j,n}) \cdot (\# \mathcal{C}_{k,n})] \notag \\
&\quad +(s-\rho) E_{\alpha}[(\# \mathcal{C}_{1,n})^2]+ 2(s - \rho)\rho \sum_{1<k \leq n} E_{\alpha}[(\# \mathcal{C}_{1,n}) \cdot (\# \mathcal{C}_{k,n})]. \label{eq:MiyazakiTakei20AppB-Vtemp}
\end{align}
The first two terms in \eqref{eq:MiyazakiTakei20AppB-Vtemp} are 
\begin{align*}
&\rho^2 E_{\alpha}\left[ \left(\sum_{j=1}^n \# \mathcal{C}_{j,n}\right)^2 \right]+(\rho-\rho^2) \sum_{j=1}^n E_{\alpha}[(\# \mathcal{C}_{j,n})^2]  \\
&=\rho^2n^2 + \rho(1-\rho) \sum_{j=1}^n E_{\alpha}[(\# \mathcal{C}_{j,n})^2].
\end{align*}
Noting that
\begin{align*}
\sum_{1<k \leq n} E_{\alpha}[(\# \mathcal{C}_{1,n}) \cdot (\# \mathcal{C}_{k,n})] 
&= E_{\alpha}\left[(\# \mathcal{C}_{1,n})\cdot \sum_{1<k \leq n} (\# \mathcal{C}_{k,n}) \right] \\
&= E_{\alpha}[(\# \mathcal{C}_{1,n})\cdot (n-\# \mathcal{C}_{1,n} )],
\end{align*}
the other two terms in \eqref{eq:MiyazakiTakei20AppB-Vtemp} are 
\begin{align*}
&(s-\rho) E_{\alpha}[(\# \mathcal{C}_{1,n})^2]+ 2(s-\rho)\rho E_{\alpha}[(\# \mathcal{C}_{1,n})\cdot (n-\# \mathcal{C}_{1,n} )] \\
&=(1-2\rho)(s-\rho) E_{\alpha}[(\# \mathcal{C}_{1,n})^2]+ 2(s-\rho)\rho n E_{\alpha}[\# \mathcal{C}_{1,n}] .
\end{align*}
Using 
\begin{align*}
(E_{p,q,s}[H_n])^2 &= \rho^2 n^2 +2(s-\rho) \rho n E_{\alpha} [\# \mathcal{C}_{1,n} ]+ (s-\rho)^2 (E_{\alpha} [\# \mathcal{C}_{1,n} ])^2,
\end{align*}
we have the conclusion.
\end{proof}

Combining \eqref{eq:MiyazakiTakei20AppB-V} with Lemma \ref{lem:Kursten16summary}, we can obtain the asymptotics of the variance. 

\begin{corollary} \label{cor:MiyazakiTakei20AppB} When $q>0$,
\begin{align*}
V_{p,q,s} [H_n] 
&\sim \begin{cases}
\dfrac{\rho (1-\rho)}{1-2\alpha}n &(\alpha<1/2), \\
\rho(1-\rho) n \log n &(\alpha=1/2), \\
\left[ \dfrac{\rho(1-\rho)}{(2\alpha-1)\Gamma(2\alpha)} + \dfrac{(1-2\rho)(s-\rho)}{\Gamma(1+2\alpha)} - \dfrac{(s-\rho)^2}{\Gamma(1+\alpha)^2}\right] n^{2\alpha} &(\alpha>1/2) \\ 
\end{cases}
\end{align*}
as $n \to \infty$.
\end{corollary}

To close this section, we give a remark on phase transition of the {\it biased elephant random walk} $\{S_n\}$ on $\mathbb{Z}$:  
\begin{itemize}
\item With probability $\alpha$, the walker repeats one of previous steps.
\item With probability $1-\alpha$, the walker performs like a simple random walk, which jumps to the right with probability $\rho$, or to the left with probability $1-\rho$ (The unbiased case $\rho=1/2$ is the original elephant random walk explained in the Introduction, where $p \geq 1/2$ and $\alpha=2p-1$.)
\end{itemize}
This is obtained from the minimal random walk model as follows: Let 
\[ Y_i :=2X_i-1 \quad \mbox{and} \quad S_n=\sum_{i=1}^n Y_i = 2H_n-n. \]
Then $P(Y_1=+1)=1- P(Y_1=-1)=s$,
and by \eqref{eq:HarbolaKumarLindenberg14PREtrans},
\[
P(Y_{n+1}=\pm 1 \mid \mathcal{F}_n) = \alpha \cdot \dfrac{\#\{i=1,\cdots,n : Y_i=\pm 1\}}{n} + (1-\alpha) \cdot \rho.
\]
By \eqref{eq:Colettietal19sLLN}, we have
\begin{align}
 \lim_{n \to \infty} \dfrac{S_n}{n}= 2\rho-1\qquad \mbox{a.s..} \label{eq:Colettietal19sLLNBiasedRW}
\end{align}

Consider bond percolation on $T_n$ with parameter $\alpha$, and assign `spin' $m_j:=2\xi_j-1\in \{+1,-1\}$ to each of percolation clusters $\mathcal{C}_{j,n}$, independently for different clusters. By \eqref{eq:Kuersten16(11)bis},
$S_n$ has the same distribution as
\begin{align*}
\sum_{j=1}^n m_j \cdot (\# \mathcal{C}_{j,n}).
\end{align*}
The above procedure is essentially the same as the {\it ``Divide and Color" model} introduced by H\"{a}ggstr\"{o}m \cite{Haggstrom01SPA}.
When $s=\rho=1/2$, the resulting model resembles the Ising model with zero external field, and increasing $\alpha$ corresponds to lowering the temperature. The parameter $\varepsilon:=2\rho-1$ plays a similar role to the external field in the Ising model. By \eqref{eq:Colettietal19sLLNBiasedRW}, when $\varepsilon \neq 0$, the asymptotic speed of the walker remains unchanged regardless of the value of $\alpha$. On the other hand, when $\varepsilon = 0$, the walker admits a phase transition from diffusive to superdiffusive behavior. This is reminiscent of the fact that the Ising model admits a phase transition only when the external field is zero.

\section{Martingale limit theorems} \label{sec:MartLimitThm}

\begin{theorem}[Hall and Heyde \cite{HallHeyde80}, Theorem 2.15] \label{thm:HallHeyde80Thm215} Suppose that $\{M_n\}$ is a square-integrable martingale with mean $0$.
Let $d_k=M_k-M_{k-1}$ for $k=1,2,\cdots$, where $M_0=0$. On the event
\[ \left\{ \sum_{k=1}^{\infty}  E[(d_k)^2 \mid \mathcal{F}_{k-1}] <+\infty \right\},  \]
$\{M_n\}$ converges a.s..
\end{theorem}

\begin{theorem}[Heyde \cite{Heyde77}, Theorem 1 (b)] \label{thm:Heyde77Theorem1b} Suppose that $\{M_n\}$ is a square-integrable martingale with mean $0$. Let $d_k=M_k-M_{k-1}$ for $k=1,2,\cdots$, where $M_0=0$. If  
\[ \displaystyle \sum_{k=1}^{\infty} E[(d_k)^2] < +\infty \]
holds in addition, then we have the following: Let 
\[
\displaystyle W_n^2 :=\sum_{k=n}^{\infty} (d_k)^2 
\quad \mbox{and} \quad 
s_n^2 := \sum_{k=n}^{\infty} E[ (d_k)^2].
\]
\begin{itemize}
\item[(i)] The limit $M_{\infty}:=\sum_{k=1}^{\infty} d_k$ exists a.s., and $M_n \stackrel{L^2}{\to} M_{\infty}$.
\item[(ii)] Assume that
\begin{itemize}  
\item[a)] $\displaystyle \dfrac{W_n^2}{s_n^2}  \to \eta^2$ as $n \to \infty$ in probability, and
\item[b)] $\displaystyle \lim_{n \to \infty} \dfrac{1}{s_n^2}E\left[  \sup_{k \geq n} (d_k)^2\right]=0$,
\end{itemize}
where $\eta^2$ is some a.s. finite and non-zero random variable.
Then we have
\begin{align*}
\dfrac{M_{\infty} - M_n}{W_{n+1}} &= \dfrac{ \sum_{k=n+1}^{\infty} d_k}{W_{n+1}}\stackrel{d}{\to} Z,\quad {and} \\
\dfrac{M_{\infty} - M_n}{s_{n+1}} &= \dfrac{ \sum_{k=n+1}^{\infty} d_k}{s_{n+1}}\stackrel{d}{\to} \widehat{\eta} \cdot Z, 
\end{align*} 
where $Z$ is distributed as $N(0,1)$, and $\widehat{\eta}$ is independent of $Z$ and distributed as $\eta$.
\item[(iii)] Assume that the following three conditions hold:
\begin{itemize}
\item[a')] $\displaystyle \dfrac{W_n^2}{s_n^2} \to \eta^2$ as $n \to \infty$ a.s.,
\item[c)] $\displaystyle \sum_{k=1}^{\infty} \dfrac{1}{s_k} E[ |d_k| : |d_k| > \varepsilon s_k] < +\infty$ for any $\varepsilon > 0$, and
\item[d)] $\displaystyle \sum_{k=1}^{\infty} \dfrac{1}{s_k^4} E[ (d_k)^4 : |d_k| \leq \delta s_k] < +\infty$ for some $\delta > 0$.
\end{itemize}
Then $\displaystyle \limsup_{n \to \infty} \pm \dfrac{M_{\infty} - M_n}{\widehat{\phi}(W_{n+1}^2)} =1$ a.s., where $\widehat{\phi}(t):=\sqrt{2t\log |\log t|}$.
\end{itemize}
\end{theorem}

\begin{remark} \label{rem:Heyde77Theorem1b} 
A sufficient condition for b) in Theorem \ref{thm:Heyde77Theorem1b} is that 
\[ \displaystyle \lim_{n \to \infty}\dfrac{1}{s_n^2}  \sum_{k=n}^{\infty} E[ (d_k)^2 : |d_k| > \varepsilon s_n] =0 \]
for any $\varepsilon > 0$. (See the proof of Corollary 1 in Heyde \cite{Heyde77}.)
\end{remark}

\section{The Mittag--Leffler disribution} \label{app:MLdistr}

The {\it Mittag--Leffler function} is defined by
\[ E_{\alpha}(z) := \sum_{k=0}^{\infty} \dfrac{z^k}{\Gamma(k + \alpha)}\qquad (\alpha,z \in \mathbb{C}). \]
Note that $E_1(z)=e^z$ (See e.g. \cite{BinghamGoldieTeugels89}, p. 315).

The random variable $X$ is {\it Mittag--Leffler distributed} with parameter $p \in [0,1]$ if 
\[ E[e^{\lambda X}] = E_p (\lambda) =  \sum_{k=0}^{\infty} \dfrac{\lambda^k}{\Gamma(1+ kp)}\qquad \mbox{for $\lambda \in \mathbb{R}$}. \] 
Thus the $k$-th moment of $X$ is $ \dfrac{k!}{\Gamma(1+ kp)}$, and this distribution is determined by moments (see \cite{BinghamGoldieTeugels89}, p. 329 and p. 391).
If $p=0$ (resp. $p=1$), then $X$ has the exponential distribution with mean one (resp. $X$ concentrates on $\{1\}$). 
For $p \in (0,1)$, the probability density function $f_p(x)$ of $X$ is 
\[ f_p(x)= \dfrac{\rho_p(x^{-1/p})}{px^{1+1/p}}\qquad \mbox{for $x>0$},  \]
where $\rho_p(x)$ is the density function of the one-sided $\mbox{stable$(p)$}$ distribution. See \cite{Pollard48BAMS} for details. In particular, $f_{1/2}$ is the density function of the standard half-normal distribution:
\[ f_{1/2} (x) = \sqrt{\dfrac{2}{\pi}} e^{-x^2/2}\qquad \mbox{for $x>0$}.  \]

%




\end{document}